\documentclass[a4paper,12pt]{amsart}

\usepackage[utf8]{inputenc}
\usepackage{amsmath, amsthm, amsfonts, amssymb}
\usepackage[foot]{amsaddr}
\usepackage[hidelinks]{hyperref} 

\newcommand{\norm}[1]{\left\Vert#1\right\Vert}

\newcommand{\eps}{\varepsilon}

\newcommand{\N}{\mathbb{N}}

\newcommand{\R}{\mathbb{R}}

\newcommand{\abs}[1]{\left|#1\right|}

\newcommand{\set}[1]{\Bigl\{#1\Bigr\}}

\newcommand{\co}{\colon}

\newcommand{\be}{\begin{equation}}
\newcommand{\ee}{\end{equation}}

\newcommand{\lin}{\ensuremath{\mathrm{lin}}}

\newtheorem{theorem}{Theorem}

\newtheorem{prop}[theorem]{Proposition}

\theoremstyle{definition}

\title[Sampling vs.~entropy]{
Sampling and entropy numbers \\ in 
the uniform norm
}

\author{Mario Ullrich}

\date{\today}

\address{
Institute of Analysis \& Department of Quantum Computing, 
Johannes Kepler University Linz, Austria}
\email{mario.ullrich@jku.at}

\allowdisplaybreaks


\begin{document}

\begin{abstract}
We prove a sharp bound between sampling numbers and entropy numbers in the uniform norm 
for bounded convex sets of bounded functions. 

\end{abstract}

\maketitle

Let $F$ be a class of 
real-valued functions on a domain $D$
and let $Y$ be some semi-normed space that contains $F$.
The \emph{$n$-th sampling number} of the class $F$ in $Y$ is defined by
\[
g_n(F,Y) \,:=\, 
\inf_{\substack{x_1,\dots,x_n\in {D}\\ 
\phi\co \R^n\to Y
}}\, 
\sup_{f\in F}\, 
\Big\|f - \phi \big(f(x_1),\dots,f(x_n)\big)\Big\|_Y.
\]
This is the \emph{minimal worst case error} that 
can be achieved 
by 
sampling recovery algorithms 
based on at most $n$ function 
evaluations 
at fixed nodes~$x_i$ 
if the error is measured in $Y$. 
We do not consider adaptive choice of $x_i$. 
Here, we only consider $Y=B(D)$, 
i.e., the space of all bounded 
real-valued functions on $D$ 
equipped with the sup-norm $\|f\|_\infty:=\sup_{x\in D}|f(x)|$.

We present an elementary argument 
that shows that 
$g_n(F,B(D))$ 
is bounded from above 
in terms of
the \emph{entropy numbers}, 
defined by
\[
 \eps_n(F,Y) \,:=\, 
 \inf_{y_1, \dots, y_{2^{n}} \in Y} \, \sup_{f\in F} \, \min_{i=1,\dots,2^n} \|f-y_i\|_Y.
\]
In essence, $\eps_n(F,Y)$ represents the smallest $\varepsilon$ such that 
$F$ can be covered by $2^{n}$ balls in $Y$ of radius $\varepsilon$. 
The entropy numbers are closely related to the \emph{covering numbers}, and a
rapid decay indicates that $F$ can be approximated by a relatively small number of elements, implying a low-dimensional structure or effective compressibility.
The $n$-th entropy number can be interpreted as the minimal worst-case error over all algorithms that have access to $n$ 
\emph{binary measurements}, 
where the $n$ bits encode the 
closest point $y_i$, 
and entropy numbers are therefore often used as benchmark for compression. 
Note that 
$\eps_n$ is bounded (or a null sequence) if and only if $F$ is bounded (or totally bounded). 

\medskip

We now state the main result, Theorem~\ref{thm:main}, and we will discuss related bounds below. 
Note that the main advances 
over previous results are its generality, 
as well as better constants and an easier proof.

\begin{theorem}\label{thm:main}
Let $D$ be a set
and $F\subset B:=B(D)$ be convex and 
bounded.
Then, for every 
$n\in\N_0$, 
we have
\[
g_{n}(F,B)  
\;\le\; (n+1)\cdot \eps_n(F,B). 
\]
\end{theorem}


\bigskip
\goodbreak

In the proof below, 
which is very much inspired by \cite[Theorem~12.3.2]{Pietsch-ideals}, 
we will also use the \emph{inner entropy numbers} 
\[
\varphi_n(F,Y) \,:=\, 
\sup_{y_1, \dots, y_{2^{n}+1} \in F}\, \min_{i\neq j}\, \frac12\,\big\|y_i-y_j\big\|_Y
\]
and the basic fact that 
$\varphi_n(F,Y) \le \eps_n(F,Y)\le 2\cdot \varphi_n(F,Y)$, 
see~\cite{Carl-Stephanie-90} 
or~\cite[Chapter~12]{Pietsch-ideals}.
Moreover, we use 
\[
g_n^{0}(F,Y) \,:=\, 
\inf_{\substack{x_1,\dots,x_n\in {D}}}\, 
\sup_{\substack{f,g\in F\co \\ f(x_k)=g(x_k), \\ k=1,\dots,n}}\, 
\big\|f - g\big\|_Y, 
\]
which is called the \emph{(global) diameter of information}, 
see~\cite[Section~4.1]{NW1}.
Again, we have the general relation 
$g_n(F,Y)\le g_n^0(F,Y) \le 2\cdot g_n(F,Y)$, 
and we will use below that in fact $g_n^0(F,B(D))= 2\cdot g_n(F,B(D))$.

\medskip

\begin{proof}[Proof of Theorem~\ref{thm:main}] 
%
We first show that, for 
every $\rho<g_{n}^0(F,B)$, 
we can find 
$f_0,g_0, \ldots, f_n,g_n \in F$ and 
$x_0,\ldots,x_n \in D$ 
such that
$f_k(x_j)=g_k(x_j)$ for $j<k$ and 
$|f_k(x_k)-g_k(x_k)|>\rho$ 
for $k=0,\ldots,n$. 

The proof is by induction. 
Let 
$k\in\{0,\dots,n\}$ 
and assume that $f_j,g_j$ and $x_j$ for $j<k$ have already been found. 
Define 
\[
M_k \,:=\, \set{(f,g)\in F\times F\co f(x_j)=g(x_j) \;\;\text{ for }\; 0\le j< k}. 
\] 
By definition, 
we can choose $f_k,g_k\in F$ 
with 
$(f_k,g_k)\in M_k$ 
and 
\begin{equation*}
\|f_k-g_k\|_\infty \,>\, \rho. 
\end{equation*}  
Hence,  
there is $x_k\in D$ with 
\begin{equation*} 
    |f_k(x_k)-g_k(x_k)| \;>\; \rho.
\end{equation*}
This finishes the induction step. 

\medskip

For 
$\xi=(\xi_0,\dots,\xi_n)\in\{0,1\}^{n+1}$, 
we now define 
\[
h_\xi:=\frac{1}{n+1} \sum_{i=0}^n \left( \xi_i f_i + (1-\xi_i) g_i \right). 
\]
By convexity, $H_n:=\{h_\xi\co \xi\in\{0,1\}^{n+1}\}\subset F$ with $\#H_n=2^{n+1}$.

For $\xi,\xi'\in H_n$, $\xi\neq\xi'$, 
let $k:=\min\{j\co \xi_j\neq\xi'_j\}$.
Then, 
\[\begin{split}
\abs{h_\xi(x_k)-h_{\xi'}(x_k)} 
\,&=\, \frac{1}{n+1} \abs{\sum_{i=k}^{n} ( \xi_i-\xi'_i) f_i(x_k) + (\xi'_i-\xi_i) g_i(x_k) } \\
\,&=\, \frac{1}{n+1} \abs{\xi_k-\xi'_k} \abs{f_k(x_k) - g_k(x_k) } 
\,\ge\,  \frac{\rho}{n+1}.
\end{split}
\]
This implies 
$\norm{h_\xi-h_{\xi'}}_\infty 
\,\ge\,  \frac{\rho}{n+1}$ 
for all $\xi\neq \xi'\in\{0,1\}^{n+1}$ 
and therefore 
\[
\eps_n(F,B) 
\,\ge\, \varphi_n(F,B) 
\,\ge\, \frac{\rho}{2(n+1)}.
\]
Since this holds for all $\rho<g_{n}^0(F,B)$, we obtain 
$g_{n}^0 \le 2(n+1)\cdot \eps_n$.

We finally show that $g_{n}^0=2\cdot g_n$. 
This is well-known, but we recall a simple proof. 
For this note that a 
reconstruction 
$\phi^*(y)\in B(D)$ for information 
$y\in\R^n$ 
is given by 
the function 
\[
\phi^*(y)(x) := 
\frac12\Biggl(\sup_{\substack{f\in F_y}} f(x)
\,+\,\inf_{\substack{f\in F_y}} f(x)\Biggr)
\]
with $F_y:=F_{y,\{x_1,\dots,x_n\}}:=\{f\in F\co f(x_k)=y_k,\; k=1,\dots,n\}$, 
and $\phi^*(y)=0$ whenever $F_y=\varnothing$. 
By boundedness, $\sup$ and $\inf$ exist. 
This is sometimes called a \emph{Chebyshev center}, 
and it is optimal in a worst case sense over~$F$.
%

With this, 
and taking into account that
$|a-b|=2\cdot|a-m|$ with $m=\frac{a+b}{2}$ for all 
$a,b\in\R$, 
we obtain 
for any choice of 
$x_1,\dots,x_n\in D$ 
that 
\[\begin{split}
g_n(F,B)
\,&\le\, 
\sup_{y\in \R^n}\,
\sup_{f\in F_y}\, 
\big\|f - \phi^*(y)\big\|_\infty 
\,=\, \frac12 
\sup_{y\in \R^n}\,
\sup_{f,g\in F_y}\, 
\big\|f - g\big\|_\infty. 
\end{split}
\]
Taking the infimum over all $x_1,\dots,x_n\in D$ proves 
$2\cdot g_n\le g_n^0$. 
The reverse inequality is obvious. 


\end{proof}

\medskip

The convexity in Theorem~\ref{thm:main} 
is needed. This follows, e.g., from results of~\cite{MUV15}, 
see Remark~4.4, that show that 
$\eps_n$ may decay much faster than~$g_n$.
%
However, it can be relaxed under another condition. \\
In fact, if $F=B_X$ is the unit ball of a $p$-Banach space $X$, $0<p\le1$, i.e., we have 
$\|f+g\|_X^p\le\|f\|_X^p+\|g\|_X^p$ for all $f,g\in X$, 
then the proof works with the only difference that 
the $(n+1)$ in the definition of $h_\xi$ must be replaced by 
$(n+1)^{1/p}$. This implies 
the following proposition. 

\begin{prop} 
Let $D$ be a set 
and $F\subset B:=B(D)$ be the unit ball of 
a $p$-Banach space with $0<p\le1$. 
Then, for every 
$n\in\N_0$, 
we have
\[
g_{n}(F,B) \;\le\; (n+1)^{1/p}\cdot \eps_n(F,B).
\]
\end{prop}

\bigskip



\medskip

If $F\subset B(D)$ is convex and symmetric, 
then Theorem~\ref{thm:main} even holds
for the \emph{linear sampling numbers} 
\begin{align*}
g_n^{\lin}(F,Y) \,:=\, 
\inf_{\substack{x_1,\dots,x_n\in {D}\\ g_1,\dots,g_n\in Y}}\, 
\sup_{f\in F}\, 
\Big\|f - \sum_{i=1}^n f(x_i)\, g_i\Big\|_Y
\end{align*}
with $Y=B(D)$, 
see~~\cite{Creutzig_Wojtaszczyk2004} or \cite[Theorem 4.8]{NW1}.  
Without symmetry, we can restrict to affine algorithms, see~\cite{Suk86}. We omit the details. 


Theorem~\ref{thm:main} is sharp up to constants, as can be seen by applying it
to the unit ball of the 
univariate Sobolev space 
$W^s_p=W^s_p([0,1])$, $s\in\N$, 
on the interval equipped with the norm 
$\|f\|_{W^s_p}:=\|f\|_p+\|f^{(s)}\|_p$. \\
The entropy numbers of these spaces 
are known 
since the 1970s to satisfy 
$\eps_n(W^s_p,L_q)\asymp n^{-s}$ 
for all $1\le p,q\le \infty$ and 
$s\in\N$ ($s>1$ for $p=1$). 
We refer to~\cite[Theorem~5.2]{Birman_Solomyak} for the upper 
and to~\cite[Theorem~3]{Triebel-entropy} for the lower bound, respectively.
See also~\cite{Edmunds_Triebel} and~\cite[6.7.8.15]{Pie07} for further results and history. \\
The sampling numbers 
satisfy $g_n(W^s_p,L_q)\asymp g_n^\lin(W^s_p,L_q)\asymp n^{-s+(1/p-1/q)_+}$ 
with $(a)_+ := \max \{a, 0 \}$, again 
for all $1\le p,q\le \infty$ and $s\in\N$ ($s>1$ for $p=1$),  
see~\cite[Section~1.3.11]{Novak} 
as well as~\cite[Theorem~23 \& Corollary~25]{NoTr06}. 
Noting that $L_\infty$ can be replaced by $B$, 
we have 
$g_n(W^s_p,B)\asymp n^{1/p}\cdot \eps_n(W^s_p,B)$. 
Hence, Theorem~\ref{thm:main} gives the correct order for $p=1$. 


%

\medskip
\goodbreak

Let us now discuss some other results on (linear) sampling numbers for general spaces. 
Here, we only consider compact 
$F\subset C(D)$, i.e., the space of continuous functions, where $D$ is a compact topological space. 
We also let $(D,\mu)$ be a Borel probability space, and $L_p:=L_p(D,\mu)$. \\
These are mild assumptions, 
which ensure that the numbers considered here are the same 
no matter if we choose $Y=B(D)$, $Y=C(D)$ or $Y=L_\infty(D)$. 
However, they can sometimes be weakened further, see the corresponding papers. 

First, 
we are only aware of one similar result that bounds sampling by entropy numbers. 
In fact, it was shown recently in~\cite[Remark~2.2]{KT25} that, 
for all convex and symmetric $F\subset C(D)$, $1\le p<\infty$ and $r\neq 1/2$, we have 
\begin{equation} \label{eq:KT25}
g_n(F, L_p) \,\le\, C'\cdot n^{-r'/p} \; 
\left(\sup_{k\le n}\; (k+1)^r\cdot\eps_k(F,B) \right)^{1/p}
\end{equation}
with $r':=\min\{r,1/2\}$ and $C'>0$ only depending on $r\neq1/2$ and $p$. 
%
This result cannot give a bound smaller than 
$n^{-1/(2p)}$, no matter how fast the $\eps_n$ decay.  
In particular, it does not help for $p=\infty$. \\
In this case, 
a similar result to Theorem~\ref{thm:main}, at least for convex and symmetric sets and in an asymptotic sense, 
could also be obtained from a variation 
of~\cite[Theorem~1.2]{KNU24}.  
In fact, it can be shown that 
\begin{equation}
g_{2n-1}^{\rm lin}(F,B) 
\;\le\; 4 n\,\bigg(\prod_{k<n} \eps_k(F,B) \bigg)^{1/n},     
\end{equation}
by replacing, in the proof of~\cite[Theorem~6.4]{KNU24}, the use of Proposition~4.4 by~\cite[12.3.1]{Pietsch-ideals}, i.e., a bound between entropy and Hilbert numbers for convex and symmetric sets. 
See also~\cite{U24} for the used bounds between s-numbers. 

\medskip 

Some other results 
use 
the 
\emph{Kolmogorov width} 
defined by 
\[
d_n(F,Y) \;:=\; \inf_{\substack{V_n\subset Y\\ \dim(V_n)=n}}\sup_{f\in F}\,\inf_{g \in V_n}\|f-g\|_Y.  
\] 
In particular, 
using a simple least squares estimator with suitably chosen points, it has been shown in 
\cite[Theorem~20]{KPUU-uniform} that 
\begin{equation}
\label{eq:KPUU}
g_{4n}^\lin(F,L_p) 
\;\le\; 84\cdot n^{(1/2-1/p)_+}\;d_n(F,B).
\end{equation}
%
The case $p=2$ was obtained earlier in~\cite[Theorem~1.1]{T20}.
For $p=\infty$, the inequality $g_n\le (n+1)\cdot d_n$ is known from~\cite[Proposition~1.2.5]{Novak}. The latter is best possible, and could be improved by~\eqref{eq:KPUU} only by allowing \emph{oversampling}. 
We refer to~\cite[Section~5]{KPUU-uniform} for details. 
Now,~\eqref{eq:KPUU} is also sharp up to constants as can be seen by the same example as above  and 
$d_n(W^s_p,B)\asymp n^{-s+(1/p-1/2)_+}$, 
see~\cite[Theorem~VII.1.1]{Pinkus85} 
and~\cite[6.7.8.13 (as well as 6.2.5.2)]{Pie07} 
for more on this and the interesting history.

More is known in the case of $L_2$-approximation. 
It was shown in~\cite{KPUU-general}, 
based on~\cite{KU1,KU2,DKU}, that, 
for some absolute constant $c\in\N$, 
we have
\[
g_{cn}^{\rm lin}(F,L_2) 
\;\le\; \frac{1}{\sqrt{n}}\sum_{k > n} \frac{ d_k(F,L_2)}{\sqrt{k}}. 
\] 
This shows that 
$g^\lin_n(F,L_2)\asymp d_n(F,L_2)$ 
whenever
$d_n(F,L_2)\asymp n^{-s}$ with $s>1/2$. 
Moreover, the optimal sampling points can be replaced by a (sometimes feasible) random choice if we allow a $\log(n)$-oversampling, see the above papers or the surveys~\cite{HKNPU2,SU23}. 
This has been extended to other $Y$ in~\cite{KPUU-general}, but only under additional assumptions on~$F$. 


Also note that entropy numbers are related to Kolmogorov width, at least if $F=B_X$ is the unit ball of a (quasi-)Banach space $X$. 
In fact, we have
$d_n(B_X,Y) \le n\cdot \eps_n(B_X,Y)$ 
for Banach spaces $X,Y$, see~\cite[Theorem~12.3.2]{Pietsch-ideals}. 
Moreover, there is \emph{Carl's inequality}~\cite{Carl81}, 
i.e., 
$$\eps_n(B_X,Y) \,\le\, c_s\, n^{-s} \cdot \sup_{k\le n}\, k^s\cdot d_k(B_X,Y)$$ with a constant $c_s>0$ only depending on $s>0$, 
which also holds for quasi-Banach spaces, see~\cite{HKV16}, 
and with $d_k$ replaced by Gelfand widths/numbers.
However, it seems that none of the above results implies another.

Let us finally mention that there has been quite some interest in 
bounding sampling numbers by \emph{best-$m$-term approximation}, 
which is another quantity aiming for nonlinear approximation 
like entropy. 
We did not find any interesting corollary from Theorem~\ref{thm:main}  in this respect and leave it for further research. 
Let us only mention a few references 
that showed superiority of nonlinear sampling algorithms~\cite{JUV,Krieg-tract,Tem25}, 
and a possibly related bound on entropy, see~\cite{T13} and~\cite[Proposition~15]{PSV24}.

\bigskip

\noindent\textbf{Acknowledgement.} 
I would like to thank 
Matthieu Dolbeault, David Krieg and Erich Novak 
for helpful comments. 
Especially, I thank Mathias Sonnleitner who observed that 
the bound of Theorem~\ref{thm:main} is sharp up to a factor of $3$. This can be deduced from computations of~\cite{Schuett84}. \\ 
This work is supported by 
the Austrian Research Promotion Agency (FFG) through the project FO999921407 (HDcode) funded by the European Union via NextGenerationEU.

\linespread{.97}


\end{document}